\documentclass[
final
]{dmtcs-episciences}

\usepackage[utf8]{inputenc}
\usepackage{subfigure}

\usepackage{amsmath,amsthm,amssymb,bbm}
\usepackage{hyperref}
\usepackage{enumerate}

\newtheorem{theorem}{Theorem}

\newtheorem{lemma}{Lemma}

\theoremstyle{remark}

\theoremstyle{definition}

\newcommand{\D}[1]{\mathop{\mathrm{d}#1}}
\DeclareMathOperator*{\Var}{Var}
\DeclareMathOperator{\fp}{fp}

\author[Aksheytha Chelikavada and Hugo Panzo]{Aksheytha Chelikavada\affiliationmark{1}
  \and Hugo Panzo\affiliationmark{2}
  }
\title{Limit theorems for fixed point biased permutations avoiding a pattern of length three}
\affiliation{
  Department of Computing and Digital Media, DePaul University, Chicago, IL, USA\\
  Department of Mathematics and Statistics, Saint Louis University, St.~Louis, MO, USA}
\begin{document}

\publicationdata{vol. 28:2}{2026}{8}{10.46298/dmtcs.15388}{2025-03-18; None}{2026-01-08}

\maketitle

\begin{abstract} \bigskip
We prove limit theorems for the number of fixed points occurring in a random pattern-avoiding permutation distributed according to a one-parameter family of biased distributions. The bias parameter exponentially tilts the distribution towards favoring permutations with more or fewer fixed points than is typical under the uniform distribution. One case we study features a phase transition where the limiting distribution changes abruptly from negative binomial to Rayleigh to normal depending on the bias parameter.
\end{abstract}


\section{Introduction}

Determining the limiting distribution of the number of fixed points in a uniformly random permutation, also known as the problem of coincidences, is a classic problem in probability first inspired by the French card game \emph{jeu de treize}. The problem was solved in the early 1700's in a series of letters between Pierre R\'{e}mond de Montmort and Nikolaus Bernoulli where it was shown that the number of fixed points converges in distribution to a Poisson random variable with parameter $1$; see \cite{coincidences} for a historical account.

In this paper we generalize this classic result in two directions simultaneously. Instead of random permutations which are uniformly distributed, we consider: 
\begin{enumerate}
	\item Random permutations that are distributed according to a certain Gibbsian one-parameter family of nonuniform biased distributions, and
	\item Random permutations that avoid a pattern of length three. 
\end{enumerate}
In most cases we are able to prove a non-Poissonian limit theorem in contrast to the classical result of Montmort and Bernoulli. Indeed, in one case we uncover a phase transition where the limiting distribution depends in a qualitative way on the bias parameter. Our main results can be seen as complementary to some earlier work where the two generalizations had been considered separately; see Miner and Pak \cite{Miner_Pak}, Mukherjee \cite{Mukherjee}, and Hoffman, Rizzolo, and Slivken \cite{HRS_1, HRS_2}. Moreover, the observed phase transition is a manifestation of a more general phenomenon related to composition schemes; see \cite{composition_schemes,q_enumeration} and references therein. We also mention the recent preprint \cite{Fulman}, which proves Poissonian limit theorems for the number of fixed points occurring in several non-Gibbsian models of nonuniform permutations.

Further motivation for our work comes from the specific nature of the one-parameter family of biased distributions that we study. This family of Gibbs measures is analogous to the classical Ewens and Mallows distributions with the role of cycles and inversions being replaced by fixed points. More precisely, if $\sigma$ is a permutation and $\fp \sigma$ denotes the number of fixed points that it has, then the probability of picking $\sigma$ under this biased distribution with bias parameter $q>0$ is proportional to $q^{\fp \sigma}$. Under this \emph{fixed point biased} distribution, permutations with many fixed points are favored when $q>1$ and those with few fixed points are favored when $0<q<1$. 

Since permutations that have many fixed points are in some sense less disordered than permutations that have few fixed points, the fixed point biased distribution allows us to introduce a \emph{presortedness bias} that is of interest to computer scientists who study sorting algorithms; see \cite{presortedness} and references therein. Another similarly motivated distribution that has appeared in the literature recently is the \emph{record biased} distribution introduced in \cite{record_biased} and further studied in \cite{Jones, Corsini, Pinsky_2, record_biased2}; see also \cite{descent_biased} for the aptly named \emph{descent biased} distribution. Moreover, the additional aspect of pattern avoidance that we consider is not merely academic---it is also related to sorting algorithms. Indeed, it is well known that a permutation can be sorted with the linear time \emph{stack sort} algorithm if and only if it avoids the pattern $231$; see Section 8.2 of \cite{permutation_text}. 

The remainder of the paper is organized as follows. In Sections \ref{sec:fpbiased} and \ref{sec:fpbiased_avoiding}, we formally define the fixed point biased distribution and also recall the notion of pattern avoidance for permutations. Our main results are listed in Section \ref{sec:main_results}. In Section \ref{sec:normalization} we prove a key lemma that will be used to prove several of our main results. The proofs of the main results are given in Section \ref{sec:proofs}. Finally, we highlight some open questions as well as topics for further study in Section \ref{sec:future}.

\subsection{Fixed point biased permutations}\label{sec:fpbiased}

Let $n\in\mathbb{N}$ and denote by $\mathcal{S}_n$ the set of all permutations of $\{1,2,\dots,n\}$. For any bias parameter $q>0$, we define $\mathbb{P}_n^q$ to be the fixed point biased probability measure on $\mathcal{S}_n$ analogous to the Ewens and Mallows measures. More precisely, we have
\begin{equation}\label{eq:biased}
	\mathbb{P}_n^q(S):=\frac{1}{Z_n(q)}\sum_{\sigma\in S}q^{\fp \sigma},~ S\subset\mathcal{S}_n,
\end{equation}
where $Z_n(q)$ is the normalization constant given by
\begin{equation}\label{eq:plain_normalization}
	Z_n(q):=\sum_{\sigma\in \mathcal{S}_n}q^{\fp \sigma}.
\end{equation}
From \eqref{eq:biased} it is clear that $\mathbb{P}_n^q$ favors permutations with many fixed points when $q>1$ and it favors permutations with few fixed points when $0<q<1$. 

Alternatively, we can express $\mathbb{P}_n^q$ in terms of the uniform measure $\mathbb{P}_n$ on $\mathcal{S}_n$ and its expectation $\mathbb{E}_n$ via the mutual absolute continuity relation
\begin{equation}\label{eq:RN}
	\D{\mathbb{P}_n^q}=\frac{q^{\fp\Pi}}{\mathbb{E}_n\left[q^{\fp\Pi}\right]}\D{\mathbb{P}_n},
\end{equation}
where $\Pi$ is a permutation distributed according to the corresponding measure. Fittingly, we denote expectation under $\mathbb{P}_n^q$ by $\mathbb{E}_n^q$. From either \eqref{eq:biased} or \eqref{eq:RN}, it follows that $\mathbb{P}_n^1=\mathbb{P}_n$, and we suppress the $1$ for notational convenience.

A related distribution is the so-called \emph{Mallows model with Hamming distance}; see \cite{Mallows_Hamming}. When the central permutation in that model is the identity, then the distribution is equivalent to \eqref{eq:biased} with $\fp \sigma$ replaced by $n-\fp \sigma$. In other words, the bias is determined by the number of nonfixed points instead of fixed points. The focus of \cite{Mallows_Hamming} is on approximate sampling algorithms, though, and no limit theorems of the kind studied in this paper are proved.

\subsection{Fixed point biased pattern-avoiding permutations}\label{sec:fpbiased_avoiding}

Our main results consider random permutations that, in addition to being biased by their number of fixed points, also avoid a single pattern of length three. Here we recall the notion of pattern avoidance for permutations. Suppose that $\sigma=\sigma_1\sigma_2\cdots\sigma_n\in\mathcal{S}_n$ and $\tau=\tau_1\tau_2\cdots\tau_m\in\mathcal{S}_m$ are permutations with $2\leq m\leq n$. We say that $\sigma$ contains $\tau$ as a pattern if there exists a subsequence $1\leq i_1<i_2<\cdots<i_m\leq n$ such that for all $1\leq j,k\leq m$, the inequality $\sigma_{i_j}<\sigma_{i_k}$ holds if and only if the inequality $\tau_j<\tau_k$ holds. In other words, $\sigma$ contains $\tau$ as a pattern if and only if $\sigma$ contains a subsequence that is in the same relative order as $\tau$. If $\sigma$ does not contain $\tau$ as a pattern, then we say that $\sigma$ avoids $\tau$. For example, $31245$ contains $213$ as a pattern, while $53421$ avoids $213$. In case $\tau\in\mathcal{S}_m$ with $m>n$, then we say that every $\sigma\in\mathcal{S}_n$ avoids $\tau$. We remark that a permutation that avoids $123\dots n$ has at most $n-1$ fixed points.

If $\tau$ is a permutation, then we denote by $\mathcal{S}_n(\tau)$ the subset of $\mathcal{S}_n$ that avoids $\tau$. Now for any bias parameter $q>0$, we can define the fixed point biased probability measure on $\mathcal{S}_n(\tau)$ by
\begin{equation}\label{eq:biased_avoiding}
	\mathbb{P}_n^{q,\tau}(S):=\frac{1}{Z_n(q,\tau)}\sum_{\sigma\in S}q^{\fp \sigma},~ S\subset\mathcal{S}_n(\tau),
\end{equation}
where $Z_n(q,\tau)$ is the normalization constant given by
\begin{equation}\label{eq:normalization}
	Z_n(q,\tau):=\sum_{\sigma\in \mathcal{S}_n(\tau)}q^{\fp \sigma}.
\end{equation}
Note that $Z_n(z,\tau)$ is well-defined for any $z\in\mathbb{C}$. Indeed, considering the normalization constant with a complex parameter will prove useful.

With a slight abuse of notation, we use $\mathbb{P}_n^\tau$ to denote the uniform measure on $\mathcal{S}_n(\tau)$. There should be no ambiguity between $\mathbb{P}_n^\tau$ and $\mathbb{P}_n^q$ since we will always use lowercase Greek letters for the pattern and lowercase Roman letters for the bias parameter. This leads to the mutual absolute continuity relation
\begin{equation}\label{eq:RN_avoiding}
	\D{\mathbb{P}_n^{q,\tau}}=\frac{q^{\fp\Pi}}{\mathbb{E}_n^\tau\left[q^{\fp\Pi}\right]}\D{\mathbb{P}_n^\tau},
\end{equation}
where $\Pi$ is a permutation distributed according to the corresponding measure. Fittingly, we denote expectation under $\mathbb{P}_n^{q,\tau}$ by $\mathbb{E}_n^{q,\tau}$. From either \eqref{eq:biased_avoiding} or \eqref{eq:RN_avoiding}, it follows that $\mathbb{P}_n^{1,\tau}=\mathbb{P}_n^\tau$, and we suppress the $1$ for notational convenience.

There is also a relation between the normalization constants appearing in \eqref{eq:normalization} and \eqref{eq:RN_avoiding} that will be exploited later on. More specifically, we have
\begin{equation}\label{eq:normalization_id}
	\mathbb{E}^{\tau}_n\left[q^{\fp \Pi}\right]=\frac{\displaystyle\sum_{\sigma\in \mathcal{S}_n(\tau)}q^{\fp \sigma}}{\displaystyle\left|\mathcal{S}_n(\tau)\right|}=\frac{Z_n(q,\tau)}{Z_n(1,\tau)}.
\end{equation}


\section{Main results}\label{sec:main_results}

Our first result gives a limit theorem for the number of fixed points in a random permutation that is distributed according to the probability measure $\mathbb{P}_n^q$ defined in \eqref{eq:biased}. It can be seen as a generalization of the classical limit theorem for uniformly random permutations described in the Introduction, whereby the number of fixed points converges in distribution to a \texttt{Poisson}$(1)$ random variable. In what follows, we write $\stackrel{d}{\to}$ to indicate convergence in distribution.

\begin{theorem}\label{thm:classic}
	Let $q>0$. Then as $n\to\infty$ we have
	\[
	\fp \Pi~\text{under}~\mathbb{P}_n^q\stackrel{d}{\to}\emph{\texttt{Poisson}}(q).
	\]
\end{theorem}

The remainder of our results are limit theorems for the number of fixed points under the probability measures $\mathbb{P}_n^{q,\tau}$ defined in \eqref{eq:biased_avoiding} which involve pattern avoidance. They can be seen as generalizations of some results of Miner and Pak \cite{Miner_Pak} and Hoffman, Rizzolo, and Slivken \cite{HRS_1,HRS_2} for pattern-avoiding permutations under the uniform measure $\mathbb{P}_n^\tau$. Our results can be grouped into two classes depending on the pattern to be avoided: $\tau\in\{123\}$ and $\tau\in\{132,321,213\}$. Along with the class $\tau\in\{231,312\}$ discussed in Section \ref{sec:231_312}, for which we do not have any results, what sets these pattern classes apart is the number of fixed points occurring in their patterns. 

\begin{theorem}\label{thm:Bernoulli}
	Let $q>0$ and let $A$ and $B$ be independent \emph{\texttt{Bernoulli}}$(\frac{q}{3+q})$ random variables. Then as $n\to\infty$ we have
	\[
	\fp \Pi~\text{under}~\mathbb{P}_n^{q,123} \stackrel{d}{\to} A+B.
	\]
\end{theorem}

The next three theorems concern permutations avoiding a single pattern $\tau$ from the set $\{132,321,213\}$. Unlike the previous result, the limiting behavior of the number of fixed points depends on the parameter $q>0$ in a qualitative manner and features a phase transition at $q=3$. In the subcritical phase $q\in(0,3)$, we have convergence to the negative binomial distribution without the need for any scaling. Recall that for $r\in\mathbb{N}$ and $p\in(0,1)$, a \texttt{NegativeBinomial}$(r,p)$ random variable $N$ has a probability mass function given by
\begin{equation}\label{eq:neg_binom}
	\mathbb{P}(N=k)=\binom{k+r-1}{k}(1-p)^k p^r,~k=0,1,2,\dots 
\end{equation}

\begin{theorem}\label{thm:subcritical}
	Let $q\in(0,3)$ and $\tau\in\{132,321,213\}$. Then as $n\to\infty$ we have 
	\[
	\fp \Pi~\text{under}~\mathbb{P}_n^{q,\tau} \stackrel{d}{\to}\emph{\texttt{NegativeBinomial}}\left(2,1-\frac{q}{3}\right).
	\]
\end{theorem}

The next result deals with the critical phase $q=3$, where we have convergence to the Rayleigh distribution after scaling by $1/\sqrt{n}$. Recall that for $\sigma>0$, a \texttt{Rayleigh}$(\sigma)$ random variable $R$ has a probability density
\begin{equation}\label{eq:Rayleigh_density}
	\mathbb{P}(R\in\D{x})=\begin{cases}
		\frac{x}{\sigma^2}e^{-\frac{x^2}{2\sigma^2}}\D{x}&\text{if }x\geq 0\\
		0&\text{if }x<0.
	\end{cases}
\end{equation}
\begin{theorem}\label{thm:Rayleigh}
	Let $\tau\in\{132,321,213\}$. Then as $n\to\infty$ we have
	\[
	\frac{\fp \Pi}{\sqrt{n}}~\text{under}~\mathbb{P}_n^{3,\tau} \stackrel{d}{\to}\emph{\texttt{Rayleigh}}\left(\frac{3}{\sqrt{2}}\right).
	\]
\end{theorem}

In the supercritical phase $q>3$, we have convergence to a standard normal random variable after the appropriate centering and scaling.
\begin{theorem}\label{thm:supercritical}
	Let $q>3$ and $\tau\in\{132,321,213\}$. Then as $n\to\infty$ we have
	\[
	\frac{\fp \Pi-\frac{q(q-3)}{(q-1)(q-2)}n}{\sqrt{\frac{2q(2q-3)}{(q-1)^2(q-2)^2}n}}~\text{under}~\mathbb{P}_n^{q,\tau} \stackrel{d}{\to}\emph{\texttt{Normal}}(0,1).
	\]
\end{theorem}


\section{Asymptotic growth of the normalization constant}\label{sec:normalization}

In this section we derive the asymptotic growth of the normalization constant $Z_n(q,\tau)$ as $n\to\infty$ when $\tau\in\{132,321,213\}$. Similarly to the limits obtained in Theorems \ref{thm:subcritical}, \ref{thm:Rayleigh}, and \ref{thm:supercritical}, the asymptotic behavior of the normalization constant depends qualitatively on the parameter $q$ and exhibits a phase transition at $q=3$. The proof of our result relies on a bivariate generating function derived by S.~Elizalde that counts $\tau$-avoiding permutations marked by their length and their number of fixed points. For any $k\geq 0$ and $n\geq 1$, define the nonnegative coefficients
\[
a_{k,n}:=\big|\{\sigma\in \mathcal{S}_n(\tau)|\fp \sigma=k\}\big|.
\]
Then Theorem 3.5 of \cite{Elizalde} implies that for any $\tau\in\{132,321,213\}$ we have
\begin{equation}\label{eq:gen_function}
	\begin{split}
		G(z,q):=1+\sum_{n=1}^\infty\sum_{\sigma\in\mathcal{S}_n(\tau)}q^{\fp \sigma}z^n&=1+\sum_{n=1}^\infty \sum_{k=0}^\infty a_{k,n}\,q^k z^n\\
		&=\frac{2}{1+2(1-q)z+\sqrt{1-4z}}.
	\end{split}
\end{equation}

In what follows, we use the symbol $\sim$ to denote asymptotic equivalence. That is, for two functions $f:\mathbb{N}\to\mathbb{R}$ and $g:\mathbb{N}\to\mathbb{R}$, we say $f\sim g$ as $n\to\infty$ if 
\[
\lim_{n\to\infty}\frac{f(n)}{g(n)}=1.
\]

\begin{lemma}\label{lem:normalization}
	Let $q>0$ and $\tau\in\{132,321,213\}$. Then as $n\to\infty$ we have
	\[
	Z_n(q,\tau)\sim
	\begin{cases}
		\displaystyle\frac{4}{(3-q)^2\sqrt{\pi}}n^{-\frac{3}{2}}4^n & \mathrm{if}~q<3\\ \\
		\displaystyle\frac{2}{\sqrt{\pi}}n^{-\frac{1}{2}}4^n & \mathrm{if}~q=3\\ \\
		\displaystyle\frac{(q-1)(q-3)}{(q-2)^2}\left(\frac{(q-1)^2}{q-2}\right)^n & \mathrm{if}~q>3.
	\end{cases}
	\]
\end{lemma}

\begin{proof}[of Lemma \ref{lem:normalization}]
	From \eqref{eq:normalization} and \eqref{eq:gen_function}, we see that 
	\[
	G(z,q)=1+\sum_{n=1}^\infty Z_n(q,\tau) z^n.
	\]
	In particular,
	\begin{equation}\label{eq:normalization_coeff}
		Z_n(q,\tau)=[z^n]G(z,q),
	\end{equation}
	hence the asymptotic of $Z_n(q,\tau)$ can be obtained via singularity analysis of the generating function $G(z,q)$ \`{a} la Figure VI.7 of \cite{Flajolet}. Towards this end, we locate the singularities $\zeta\in\mathbb{C}$ of $z\mapsto G(z,q)$ that are closest to the origin, namely, the \emph{dominant singularities}. Examining the right-hand side of \eqref{eq:gen_function}, we see that there is always a branch point singularity at $\frac{1}{4}$. However, depending on the parameter $q$, there may be a pole singularity at some $\zeta$ with $|\zeta|\leq\frac{1}{4}$. More precisely, under the constraints $q>0$ and $|z|\leq\frac{1}{4}$, the equation 
	\[
	1+2(1-q)z+\sqrt{1-4z}=0
	\]
	has a solution only when $q\geq 3$, in which case the solution is $z=\frac{q-2}{(q-1)^2}$. Thus for $q\in (0,3]$, the dominant singularity is $\zeta=\frac{1}{4}$, while for $q>3$ it is $\zeta=\frac{q-2}{(q-1)^2}$.\\
	
	\noindent $\mathbf{q\in(0,3)}$ \textbf{case where} $\mathbf{\zeta=\frac{1}{4}}$
	
	In this case, $z\mapsto G(z,q)$ is analytic in $\mathbb{C}\setminus [\frac{1}{4},\infty)$ and has singular expansion 
	\begin{align*}
		G(z,q)&=\frac{4}{3-q}-\frac{16}{(3-q)^2}\sqrt{\zeta-z}+O(\zeta-z)~\mathrm{as}~z\to\zeta\\
		&=\frac{4}{3-q}-\frac{16\sqrt{\zeta}}{(3-q)^2}(1-z/\zeta)^\frac{1}{2}+O(1-z/\zeta)~\mathrm{as}~z\to\zeta.
	\end{align*}
	Hence it follows from Theorem VI.4 of \cite{Flajolet} that as $n\to\infty$ we have
	\begin{align*}
		[z^n]G(z,q)&\sim -\frac{16\sqrt{\zeta}}{(3-q)^2}\frac{n^{-\frac{1}{2}-1}}{\Gamma(-1/2)}\zeta^{-n}\\
		&=\frac{4}{(3-q)^2\sqrt{\pi}}n^{-\frac{3}{2}}4^n.
	\end{align*}
	
	\noindent $\mathbf{q=3}$ \textbf{case where} $\mathbf{\zeta=\frac{1}{4}}$
	
	In this case, $z\mapsto G(z,3)$ is analytic in $\mathbb{C}\setminus [\frac{1}{4},\infty)$ and has singular expansion
	\begin{align*}
		G(z,3)&=\frac{1}{\sqrt{\zeta-z}}+O(1)~\mathrm{as}~z\to\zeta\\
		&=\frac{1}{\sqrt{\zeta}}(1-z/\zeta)^{-\frac{1}{2}}+O(1)~\mathrm{as}~z\to\zeta.
	\end{align*}
	Hence it follows from Theorem VI.4 of \cite{Flajolet} that as $n\to\infty$ we have
	\begin{align*}
		[z^n]G(z,3)&\sim \frac{1}{\sqrt{\zeta}}\frac{n^{\frac{1}{2}-1}}{\Gamma(1/2)}\zeta^{-n}\\
		&=\frac{2}{\sqrt{\pi}}n^{-\frac{1}{2}}4^n.
	\end{align*}
	
	\noindent $\mathbf{q>3}$ \textbf{case where} $\mathbf{\zeta=\frac{q-2}{(q-1)^2}}$
	
	In this case, $z\mapsto G(z,q)$ is analytic in $\mathbb{C}\setminus [\frac{q-2}{(q-1)^2},\infty)$ and has singular expansion
	\begin{align*}
		G(z,q)&=\frac{q-3}{(q-1)(q-2)}\frac{1}{\zeta-z}+O(1)~\mathrm{as}~z\to\zeta\\
		&=\frac{q-3}{(q-1)(q-2)\zeta}(1-z/\zeta)^{-1}+O(1)~\mathrm{as}~z\to\zeta.
	\end{align*}
	Hence it follows from Theorem VI.4 of \cite{Flajolet} that as $n\to\infty$ we have
	\begin{align*}
		[z^n]G(z,q)&\sim \frac{q-3}{(q-1)(q-2)\zeta}\frac{n^{1-1}}{\Gamma(1)}\zeta^{-n}\\
		&=\frac{(q-1)(q-3)}{(q-2)^2}\left(\frac{(q-1)^2}{q-2}\right)^n.
	\end{align*}
\end{proof}


\section{Proofs of the main results}\label{sec:proofs}

Here we collect the proofs of our main results. Theorem \ref{thm:Bernoulli} is proved using an existing limit theorem for the unbiased distribution $\mathbb{P}^{123}_n$. The other theorems are proved using tools from analytic combinatorics such as singularity analysis of generating functions.

\subsection[Proof of Theorem 1]{Proof of Theorem \ref{thm:classic}}

We prove Theorem \ref{thm:classic} using the exponential generating function 
\begin{equation}\label{eq:egf}
	H(t,u):=1+\sum_{n=1}^\infty\sum_{\sigma\in \mathcal{S}_n}u^{\fp \sigma}\,\frac{t^n}{n!}=\frac{e^{t(u-1)}}{1-t};
\end{equation}
see \cite[Example 3.70]{permutation_text}. Notice that for $u=q>0$, the power series coefficients of $t^n$ in \eqref{eq:egf} are $Z_n(q)/n!$, where $Z_n(q)$ is the normalization constant defined in \eqref{eq:plain_normalization}. We will also need the well-known formula for the probability generating function of a \texttt{Poisson}$(\lambda)$ random variable $X$, namely, for any $z\in\mathbb{C}$ we have
\begin{equation}\label{eq:Poisson_pgf}
	\mathbb{E}\left[z^X\right]=e^{\lambda(z-1)}.
\end{equation}

\begin{proof}[of Theorem \ref{thm:classic}]
	Let $z\in\mathbb{C}$ and use \eqref{eq:biased} to write
	\begin{align}
		\mathbb{E}_n^q\left[z^{\fp \Pi}\right]&=\frac{1}{Z_n(q)}\sum_{\sigma\in\mathcal{S}_n}z^{\fp \sigma}q^{\fp \sigma} \nonumber \\
		&=\frac{[t^n]H(t,zq)}{[t^n]H(t,q)}.\label{eq:plain_pgf}
	\end{align}
	
	Next we need the asymptotic growth of the coefficients of $H(t,u)$. We can obtain this asymptotic by applying singularity analysis to the function $t\mapsto H(t,u)$, which from the right-hand side of \eqref{eq:egf}, is seen to be meromorphic on $\mathbb{C}$ with a single pole at $t=1$. Moreover, the singular expansion of $t\mapsto H(t,u)$ is
	\begin{align*}
		H(t,u)=e^{u-1}(1-t)^{-1}+O(1)~\mathrm{as}~t\to 1.
	\end{align*}
	Hence it follows from Theorem VI.4 of \cite{Flajolet} that as $n\to\infty$ we have
	\begin{align*}
		[t^n]H(t,u)&\sim e^{u-1}\frac{n^{1-1}}{\Gamma(1)}1^{-n}\\
		&=e^{u-1}.
	\end{align*}
	We can use this asymptotic along with \eqref{eq:plain_pgf} to write
	\begin{align*}
		\lim_{n\to\infty}\mathbb{E}_n^q\left[z^{\fp \Pi}\right]&=\frac{e^{zq-1}}{e^{q-1}}\\
		&=e^{q(z-1)}.
	\end{align*}
	
	Now the theorem follows from \eqref{eq:Poisson_pgf} and the fact that pointwise convergence of probability generating functions on a subset of $\mathbb{C}$ that contains a limit point in the interior of the unit disk implies convergence in distribution of the corresponding discrete random variables; see Theorem IX.1 of \cite{Flajolet} for a precise statement of this fact. 
\end{proof}

\subsection[Proof of Theorem 2]{Proof of Theorem \ref{thm:Bernoulli}}

Unlike the previous proof, a generating function analogous to \eqref{eq:egf} is unavailable, hence in this case we must resort to using an existing limit theorem for unbiased permutations. More specifically, $\fp\Pi$ under $\mathbb{P}_n^{123}$ converges in distribution to the sum of two independent \texttt{Bernoulli}$(\frac{1}{4})$ random variables as $n\to\infty$; see Theorem 1.1 of \cite{HRS_1}. We will also need the formula for the probability generating function of the sum of two independent \texttt{Bernoulli}$(p)$ random variables $A$ and $B$, that is, for any $z\in\mathbb{C}$ we have
\begin{align}
	\mathbb{E}\left[z^{A+B}\right]&=\Big(\mathbb{E}\left[z^A\right]\Big)^2\nonumber \\
	&=\big(1+p(z-1)\big)^2.\label{eq:Bernoulli_pgf}
\end{align}

\begin{proof}[of Theorem \ref{thm:Bernoulli}]
	Let $z\in\mathbb{C}$ and $A$ and $B$ be two independent \texttt{Bernoulli}$(\frac{1}{4})$ random variables. Since $x\mapsto q^x$ and $x\mapsto (zq)^x$ are both bounded continuous functions on $\{0,1,2\}$, and $\fp\Pi$ under $\mathbb{P}_n^{123}$ converges in distribution to the sum of two independent \texttt{Bernoulli}$(\frac{1}{4})$ random variables as $n\to\infty$, we can use the mutual absolute continuity relation \eqref{eq:RN_avoiding} along with \eqref{eq:Bernoulli_pgf} and the definition of convergence in distribution to write
	\begin{align*}
		\lim_{n\to\infty}\mathbb{E}^{q,123}_n\left[z^{\fp \Pi}\right]&=\lim_{n\to\infty}\frac{\mathbb{E}^{123}_n\left[z^{\fp \Pi}q^{\fp \Pi}\right]}{\mathbb{E}^{123}_n\left[q^{\fp \Pi}\right]}\\
		&=\frac{\mathbb{E}\left[(zq)^{A+B}\right]}{\mathbb{E}\left[q^{A+B}\right]}\\
		&=\frac{\left(1+\frac{1}{4}(zq-1)\right)^2}{\left(1+\frac{1}{4}(q-1)\right)^2}\\
		&=\left(1+\frac{q}{3+q}(z-1)\right)^2.
	\end{align*}
	
	Now the theorem follows from \eqref{eq:Bernoulli_pgf} and the fact that pointwise convergence of probability generating functions on a subset of $\mathbb{C}$ that contains a limit point in the interior of the unit disk implies convergence in distribution of the corresponding discrete random variables; see Theorem IX.1 of \cite{Flajolet} for a precise statement of this fact. 
\end{proof}

\subsection[Proofs of Theorems 3, 4, and 5]{Proofs of Theorems \ref{thm:subcritical}, \ref{thm:Rayleigh}, and \ref{thm:supercritical}}

When $\tau\in\{132,321,213\}$, the presence of a phase transition at $q=3$ requires three different proof techniques depending on whether $q<3$, $q=3$, or $q>3$.

\subsubsection[Subcritical phase q<3]{Subcritical phase $q<3$}

Our proof of Theorem \ref{thm:subcritical} is similar in spirit to that of Theorem \ref{thm:classic}, and uses the asymptotic of the normalization constant which we derived in Lemma \ref{lem:normalization}. We will also need the probability generating function of a \texttt{NegativeBinomial}$(2,p)$ random variable $N$ with $p\in (0,1)$. For any $z\in\mathbb{C}$ with $|z|<\frac{1}{1-p}$, we can use \eqref{eq:neg_binom} to write 
\begin{align}
	\mathbb{E}\left[z^N\right]&=\sum_{k=0}^\infty \binom{k+1}{k}(1-p)^k p^2 z^k\nonumber \\
	&=p^2\sum_{k=1}^\infty k\big((1-p)z\big)^{k-1}\nonumber \\
	&=\frac{p^2}{\big(1-(1-p)z\big)^2}.\label{eq:nbinom_pgf}
\end{align}

\begin{proof}[of Theorem \ref{thm:subcritical}]
	For any $u\in (0,3)$, we can use \eqref{eq:normalization_id} along with Lemma \ref{lem:normalization} to deduce that
	\begin{align}
		\lim_{n\to\infty}\mathbb{E}^{\tau}_n\left[u^{\fp \Pi}\right]&=\lim_{n\to\infty}\frac{Z_n(u,\tau)}{Z_n(1,\tau)}\nonumber \\
		&=\frac{4}{(3-u)^2}.\label{eq:E_partition}
	\end{align}
	
	Put $p=1-\frac{q}{3}\in (0,1)$. For any $z\in (0,1)$, we can use the mutual absolute continuity relation \eqref{eq:RN_avoiding} along with \eqref{eq:E_partition} to write
	\begin{align*}
		\lim_{n\to\infty}\mathbb{E}^{q,\tau}_n\left[z^{\fp \Pi}\right]&=\lim_{n\to\infty}\frac{\mathbb{E}^{\tau}_n\left[z^{\fp \Pi}q^{\fp \Pi}\right]}{\mathbb{E}^{\tau}_n\left[q^{\fp \Pi}\right]}\\
		&=\frac{(3-q)^2}{(3-zq)^2}\\
		&=\frac{p^2}{(1-(1-p)z)^2}.
	\end{align*}
	
	Now the theorem follows from \eqref{eq:nbinom_pgf} and the fact that pointwise convergence of probability generating functions on a subset of $\mathbb{C}$ that contains a limit point in the interior of the unit disk implies convergence in distribution of the corresponding discrete random variables; see Theorem IX.1 of \cite{Flajolet} for a precise statement of this fact. 
\end{proof}

\subsubsection[Critical phase q=3]{Critical phase $q=3$}

To prove Theorem \ref{thm:Rayleigh}, we employ the method of ``moment pumping'' to identify the limiting random variable via the sequence of its positive integer moments. Accordingly, we need to compute the moments of the Rayleigh distribution. If $R$ is a \texttt{Rayleigh}$(\sigma)$ random variable with $\sigma>0$ and if $p>-2$, then the density \eqref{eq:Rayleigh_density} along with the change of variables $x\mapsto \sigma\sqrt{2x}$ readily yield
\begin{align}
	\mathbb{E}\left[R^p\right]&=\int_0^\infty x^p \frac{x}{\sigma^2}e^{-\frac{x^2}{2\sigma^2}}\D{x}\nonumber \\
	&=\left(\sigma\sqrt{2}\right)^p \int_0^\infty x^\frac{p}{2} e^{-x}\D{x}\nonumber \\
	&=\left(\sigma\sqrt{2}\right)^p\Gamma(p/2+1).\label{eq:Rayleigh_moments}
\end{align}

We will also need the fact that the \texttt{Rayleigh}$(\sigma)$ distribution is determined by its positive integer moments. A sufficient condition for this is that 
\[
\frac{\mathbb{E}\left[R^{m+1}\right]}{\mathbb{E}\left[R^m\right]}=O\left((m+1)^2\right)~\mathrm{as}~m\to\infty;
\]
see Theorem 1 of \cite{moment_determinacy}. We can check that this condition holds by using \eqref{eq:Rayleigh_moments} along with Wendel's inequality \cite[Equation 7]{Wendel} to deduce that for all $m\geq 0$
\begin{align*}
	\frac{\mathbb{E}\left[R^{m+1}\right]}{\mathbb{E}\left[R^m\right]}&=\sigma\sqrt{2}\frac{\Gamma(\frac{m+2}{2}+\frac{1}{2})}{\Gamma(\frac{m+2}{2})}\\
	&\leq \sigma\sqrt{m+2}.
\end{align*}

As an alternative to moment pumping, we could instead prove Theorem \ref{thm:Rayleigh} by appealing to \cite[Theorem 1]{Drmota_Soria}. This result is similar to \cite[Theorem IX.9]{Flajolet}, which we use to prove Theorem \ref{thm:supercritical} below. Both results provide sufficient conditions on a bivariate generating function that ensure weak convergence of a sequence of renormalized discrete random variables whose probability mass functions are derived from the coefficients of the generating function.

\begin{proof}[of Theorem \ref{thm:Rayleigh}]
	For $m\in\mathbb{N}$, let $(x)_m$ denote the falling factorial, that is
	\[
	(x)_m:=\underbrace{x(x-1)(x-2)\cdots(x-m+1)}_{\displaystyle m~\mathrm{factors}}.
	\]
	Notice that by repeated differentiation of \eqref{eq:gen_function}, we have for any $m\in\mathbb{N}$ 
	\begin{align}
		G_m(z,q):&=\sum_{n=1}^\infty\sum_{\sigma\in\mathcal{S}_n(\tau)}(\fp\sigma)_m\, q^{\fp \sigma}z^n\label{eq:Gm_gen} \\
		&=q^m\left.\frac{\partial^m G(z,x)}{\partial x^m}\right|_{x=q}\nonumber \\
		&=\frac{2\,m!(2qz)^m }{\left(1+2(1-q)z+\sqrt{1-4z}\right)^{m+1}}.\label{eq:diff_gen}
	\end{align}
	
	Similarly to the case of $G(z,3)$ covered in Section \ref{sec:normalization}, we see from \eqref{eq:diff_gen} that $z\mapsto G_m(z,3)$ has a single dominant singularity at $\zeta=\frac{1}{4}$, is analytic in $\mathbb{C}\setminus [\frac{1}{4},\infty)$, and has singular expansion
	\begin{align*}
		G_m(z,3)&=2\,m!\left(\frac{3}{2}\right)^m \frac{1}{2^{m+1}}\frac{1}{(\zeta-z)^\frac{m+1}{2}}+O\left(\frac{1}{(\zeta-z)^\frac{m}{2}}\right)~\mathrm{as}~z\to\zeta\\
		&=m!\left(\frac{3}{4}\right)^m \frac{1}{\zeta^\frac{m+1}{2}}(1-z/\zeta)^{-\frac{m+1}{2}}+O\left((1-z/\zeta)^{-\frac{m}{2}}\right)~\mathrm{as}~z\to\zeta.
	\end{align*}
	Hence it follows from Theorem VI.4 of \cite{Flajolet} that as $n\to\infty$ we have
	\begin{align}
		[z^n]G_m(z,3)&\sim m!\left(\frac{3}{4}\right)^m \frac{1}{\zeta^\frac{m+1}{2}}\frac{n^{\frac{m+1}{2}-1}}{\Gamma\left(\frac{m+1}{2}\right)}\zeta^{-n}\nonumber \\
		&=2\left(\frac{3}{2}\right)^m\frac{\Gamma(m+1)}{\Gamma\left(\frac{m+1}{2}\right)}n^{\frac{m-1}{2}}4^n\nonumber \\
		&=\frac{2}{\sqrt{\pi}}3^m \Gamma(m/2+1)n^{\frac{m-1}{2}}4^n,\label{eq:Gm_asymptotic}
	\end{align}
	where the last equality follows from the Legendre duplication formula.

	Let $X_n$ be the number of fixed points in an $\mathcal{S}_n(\tau)$-valued random permutation distributed according to $\mathbb{P}_n^{3,\tau}$. We start by computing the asymptotics of the \emph{factorial moments} of $X_n$. For any $m\in\mathbb{N}$, we can use \eqref{eq:biased_avoiding}, \eqref{eq:Gm_gen}, and \eqref{eq:normalization_coeff} to write
	
	\begin{align*}
		\mathbb{E}\left[(X_n)_m\right]=\mathbb{E}_n^{3,\tau}\left[(\fp\Pi)_m\right]&=\frac{1}{Z_n(3,\tau)}\sum_{\sigma\in\mathcal{S}_n(\tau)}(\fp\sigma)_m\, 3^{\fp\sigma}\\
		&=\frac{[z^n]G_m(z,3)}{Z_n(3,\tau)}.
	\end{align*}
	Hence it follows from Lemma \ref{lem:normalization} and \eqref{eq:Gm_asymptotic} that 
	\begin{equation}\label{eq:factorial_asymptotic}
		\mathbb{E}\left[(X_n)_m\right]\sim 3^m \Gamma(m/2+1)n^\frac{m}{2}~\mathrm{as}~n\to\infty.
	\end{equation}
	
	We can deduce the asymptotics of the positive integer moments of $X_n$ from \eqref{eq:factorial_asymptotic} by arguing inductively. Note that it follows immediately from \eqref{eq:factorial_asymptotic} that 
	\[
	\mathbb{E}\left[X_n\right]=\mathbb{E}\left[(X_n)_1\right]\sim 3\, \Gamma(1/2+1)n^\frac{1}{2}~\mathrm{as}~n\to\infty.
	\]
	For any $m\in\mathbb{N}$, assume that for all $1\leq j\leq m$ we have
	\begin{equation}\label{eq:induction_assumption}
		\mathbb{E}\left[X_n^j\right]\sim 3^j \Gamma(j/2+1)n^\frac{j}{2}~\mathrm{as}~n\to\infty.
	\end{equation}
	Then by expanding $(X_n)_{m+1}$ into a polynomial and using \eqref{eq:induction_assumption}, we can write
	\begin{equation}\label{eq:comp_asymptotic}
		\mathbb{E}\left[(X_n)_{m+1}\right]= \mathbb{E}\left[X_n^{m+1}\right]+O(n^\frac{m}{2})\mathrm{as}~n\to\infty.
	\end{equation}
	Now the asymptotics \eqref{eq:factorial_asymptotic} and \eqref{eq:comp_asymptotic} together imply that
	\begin{equation*}
		\mathbb{E}\left[X_n^{m+1}\right]\sim 3^{m+1} \Gamma\big((m+1)/2+1\big)n^\frac{m+1}{2}~\mathrm{as}~n\to\infty.
	\end{equation*}
	It follows inductively that $\mathbb{E}[(X_n)_m]$ and $\mathbb{E}[X_n^m]$ have the same asymptotic for all $m\in\mathbb{N}$. In particular, this allows us to use \eqref{eq:factorial_asymptotic} to conclude that 
	\begin{align*}
		\lim_{n\to\infty}\mathbb{E}_n^{3,\tau}\left[\left(\frac{\fp\Pi}{\sqrt{n}}\right)^m\right]&=\lim_{n\to\infty}\mathbb{E}\left[\left(\frac{X_n}{\sqrt{n}}\right)^m\right]\\
		&=3^m \Gamma(m/2+1).
	\end{align*}
	
	Since these match the positive integer moments of the \texttt{Rayleigh}$(\frac{3}{\sqrt{2}})$ distribution computed in \eqref{eq:Rayleigh_moments}, and that distribution is determined by its positive integer moments, it follows from the moment convergence theorem \cite[Theorem C.2]{Flajolet} that $\frac{\fp \Pi}{\sqrt{n}}$ under $\mathbb{P}_n^{3,\tau}$ converges in distribution to a \texttt{Rayleigh}$(\frac{3}{\sqrt{2}})$ random variable.
\end{proof}

\subsubsection[Supercritical phase q>3]{Supercritical phase $q>3$}

To prove Theorem \ref{thm:supercritical}, we use Theorem IX.9 of \cite{Flajolet}. We state this result below for the convenience of the reader. It requires the following notations: given a function $f:\mathbb{C}\to\mathbb{C}$ that is analytic at $1$ and assumed to satisfy $f(1)\neq 0$, we set
\begin{equation}\label{eq:mean_var}
	\mathfrak{m}(f)=\frac{f'(1)}{f(1)},\hspace{6mm}\mathfrak{v}(f)=\frac{f''(1)}{f(1)}+\frac{f'(1)}{f(1)}-\left(\frac{f'(1)}{f(1)}\right)^2.
\end{equation}

\begin{theorem}{\cite[Theorem IX.9]{Flajolet}}\label{thm:Flajolet}
	Let $F(z,u)$ be a function that is bivariate analytic at $(z,u)=(0,0)$ and has nonnegative coefficients. Assume that $z\mapsto F(z,1)$ is meromorphic in $|z|\leq r$ with only a simple pole at $z=\rho$ for some positive $\rho<r$. Assume also the following conditions.
	
	\begin{enumerate}
		
		\item [i)]\emph{\textbf{Meromorphic perturbation:}} there exists $\epsilon>0$ and $r>\rho$ such that in the domain\\ $\mathcal{D}=\{|z|\leq r\}\times \{|u-1|<\epsilon\}$, the function $F(z,u)$ admits the representation 
		\[
		F(z,u)=\frac{B(z,u)}{C(z,u)},
		\]
		where $B(z,u)$, $C(z,u)$ are analytic for $(z,u)\in\mathcal{D}$ with $B(\rho,1)\neq 0$. \big(Thus $\rho$ is a simple zero of $C(z,1)$.\big)
		
		\item [ii)]\emph{\textbf{Nondegeneracy:}} one has $\partial_z C(\rho,1)\cdot\partial_u C(\rho,1)\neq 0$, ensuring the existence of a nonconstant $\rho(u)$ analytic at $u=1$, such that $C(\rho(u),u)=0$ and $\rho(1)=\rho$.
		
		\item [iii)]\emph{\textbf{Variability:}} one has
		\[
		\mathfrak{v}\left(\frac{\rho(1)}{\rho(u)}\right)\neq 0.
		\]
		
	\end{enumerate}
	Then, the random variable $X_n$ with probability generating function 
	\[
	p_n(u)=\frac{[z^n]F(z,u)}{[z^n]F(z,1)}
	\]
	after standardization, converges in distribution to a Gaussian variable, with a speed of convergence that is $O(n^{-1/2})$. Moreover, as $n\to\infty$, the mean and variance of $X_n$ are of the form
	\[
	\mathbb{E}[X_n]=\mathfrak{m}\left(\frac{\rho(1)}{\rho(u)}\right)n+O(1),\hspace{6mm} \Var(X_n)=\mathfrak{v}\left(\frac{\rho(1)}{\rho(u)}\right)n+O(1).
	\]
\end{theorem}

\begin{proof}[of Theorem \ref{thm:supercritical}]
	Let $X_n$ be the number of fixed points in an $\mathcal{S}_n(\tau)$-valued random permutation distributed according to $\mathbb{P}_n^{q,\tau}$. Then for any $u\in\mathbb{C}$, we can use the mutual absolute continuity relation \eqref{eq:RN_avoiding} along with \eqref{eq:normalization_id} and \eqref{eq:normalization_coeff} to write
	\begin{align}
		\mathbb{E}\left[u^{X_n}\right]=\mathbb{E}_n^{q,\tau}\left[u^{\fp \Pi}\right]&=\frac{\mathbb{E}^{\tau}_n\left[(qu)^{\fp \Pi}\right]}{\mathbb{E}^{\tau}_n\left[q^{\fp \Pi}\right]}\nonumber \\
		&=\left.\frac{Z_n(qu,\tau)}{Z_n(1,\tau)}\middle/ \frac{Z_n(q,\tau)}{Z_n(1,\tau)}\right.\nonumber\\
		&=\frac{[z^n]G(z,qu)}{[z^n]G(z,q)},\label{eq:super_pgf}
	\end{align}
	where $G(z,u)$ is the bivariate generating function that was defined in \eqref{eq:gen_function}. Next we define $F(z,u):=G(z,qu)$. Then it is evident from \eqref{eq:gen_function} that
	\begin{equation}\label{eq:F_gen}
		\begin{split}
			F(z,u)&=1+\sum_{n=1}^\infty \sum_{k=0}^\infty q^k a_{k,n}\,u^k z^n\\
			&=\frac{2}{1+2(1-qu)z+\sqrt{1-4z}}.
		\end{split}
	\end{equation}
	Now it follows from \eqref{eq:super_pgf} that the probability generating function of $X_n$ is given by
	\[
	\mathbb{E}\left[u^{X_n}\right]=\frac{[z^n]F(z,u)}{[z^n]F(z,1)}.
	\]
	
	In order to invoke Theorem \ref{thm:Flajolet} and conclude that $X_n$ converges to a normal random variable after standardization, we must check that the conditions on $F(z,u)$ are satisfied. Towards this end, note that it follows from \eqref{eq:F_gen} that $F(z,u)$ is analytic at $(0,0)$ and has nonnegative coefficients. Moreover, during the proof of the $q>3$ case of Lemma \ref{lem:normalization}, it was established that $z\mapsto F(z,1)$ is meromorphic in $|z|\leq \frac{1}{4}$ with only a simple pole at $\frac{q-2}{(q-1)^2}\in (0,\frac{1}{4})$. \\
	
	\noindent \textbf{Condition \emph{i)}}
	
	Put $B(z,u)\equiv 2$ and $C(z,u)=1+2(1-qu)z+\sqrt{1-4z}$. Then the required representation holds in $\mathcal{D}$ with any $\epsilon>0$, $r=\frac{1}{4}$, and $\rho=\frac{q-2}{(q-1)^2}$. \\
	
	\noindent \textbf{Condition \emph{ii)}}
	
	By routine calculations, we deduce that
	\begin{align*}
		\partial_z C(\rho,1)\cdot\partial_u C(\rho,1)&=\left(2(1-q)-\frac{2}{\sqrt{1-4\rho}}\right)\cdot -2q\rho\\
		&=-2\frac{(q-1)(q-2)}{q-3}\cdot -2q\frac{q-2}{(q-1)^2}\\
		&\neq 0.
	\end{align*}
	In particular, for $|u-1|<1-\frac{3}{q}$ we have $\rho(u)=\frac{qu-2}{(qu-1)^2}$.\\
	
	\noindent \textbf{Condition \emph{iii)}}
	
	With $\rho(u)=\frac{qu-2}{(qu-1)^2}$, it is straightforward to use \eqref{eq:mean_var} to check that 
	\[
	\mathfrak{v}\left(\frac{\rho(1)}{\rho(u)}\right)=\frac{2q(2q-3)}{(q-1)^2(q-2)^2}\neq 0.
	\]
	
	Now that the conditions have been verified, we can conclude that as $n\to\infty$ 
	\[
	\mathbb{E}[X_n]=\frac{q(q-3)}{(q-1)(q-2)}n+O(1),\hspace{6mm} \Var(X_n)=\frac{2q(2q-3)}{(q-1)^2(q-2)^2}n+O(1),
	\]
	and
	\[
	\frac{\displaystyle X_n-\frac{q(q-3)}{(q-1)(q-2)}n}{\displaystyle\sqrt{\frac{2q(2q-3)}{(q-1)^2(q-2)^2}n}}\stackrel{d}{\to}\texttt{Normal}(0,1).
	\]
\end{proof}


\section{Future directions}\label{sec:future}

Here we discuss a few open questions and suggest topics for further study.

\subsection[231 or 312 avoiding case]{$\tau\in\{231,312\}$ case}\label{sec:231_312}

When proving our limit theorems for the number of fixed points occurring in a random permutation distributed according to $\mathbb{P}_n^{q,\tau}$, we relied on either an existing limit theorem for the unbiased distribution $\mathbb{P}_n^\tau$, or an explicit bivariate generating function that counts $\tau$-avoiding permutations marked by their length and their number of fixed points. The first approach no longer works in the $\tau\in\{231,312\}$ case because the existing limit theorem for $\mathbb{P}_n^\tau$ requires scaling; see Theorem 1.1 of \cite{HRS_1}. Moreover, the second approach fails because the generating function equivalent to \eqref{eq:gen_function} is known only in the form of a continued fraction which is not amenable to singularity analysis when $q\neq 1$; see Theorem 3.7 of \cite{Elizalde}.

Even if proving a limit theorem for the number of fixed points under $\mathbb{P}_n^{q,\tau}$ is out of reach in this case, it would be interesting to see if there is a phase transition similar to what occurs in the $132$, $321$, or $213$ avoiding case. We briefly sketch an heuristic argument for why such a phase transition should occur. Suppose that $\sigma\in\{231,312\}$ and $\tau\in\{132,321,213\}$. Since avoiding a pattern that has zero fixed points should result in a permutation that has more fixed points than is typical, and since $\sigma$ has zero fixed points while $\tau$ has one fixed point, a uniformly random permutation that avoids $\sigma$ should have, on average, more fixed points than a uniformly random permutation that avoids $\tau$. Indeed, this is the case as the average number of fixed points under $\mathbb{P}_n^\sigma$ is $\Theta(n^\frac{1}{4})$ while under $\mathbb{P}_n^\tau$ it is $\Theta(1)$; see \cite[Theorem 1.1]{HRS_1} and Theorem \ref{thm:subcritical}, respectively. This domination should also hold when $q>1$. In light of Theorem \ref{thm:supercritical}, this would imply that the asymptotic of $\mathbb{E}_n^{q,\sigma}[\fp \Pi]$ must transition from $\Theta(n^\frac{1}{4})$ for $q=1$ to $\Theta(n)$ for $q>3$.

\subsection[Other statistics of fix point biased permutations]{Other statistics under $\mathbb{P}_n^q$ and $\mathbb{P}_n^{q,\tau}$}

Determining the limiting distribution of fixed points under $\mathbb{P}_n^q$ and $\mathbb{P}_n^{q,\tau}$ offers only a small picture of the behavior of the whole permutation as $n\to\infty$. We can shed more light on what is happening by looking at additional permutation statistics. It turns out that in some cases, multivariate generating functions are available which count fixed points jointly with excedances and descents; see \cite{Elizalde}. When paired with the techniques used in the present paper, these generating functions should yield the limiting joint distribution of the corresponding statistics.

In contrast to this piecemeal approach, the \emph{permuton} is a more holistic way of characterizing the limiting behavior of a random permutation. This probability measure on the unit square with uniform marginals encodes many of the salient features of the random permutation; see \cite{Mukherjee, permuton} and references therein. The Mallows and record biased distributions are two examples similar to the fixed point biased distribution where the corresponding permuton has been calculated; see \cite{Starr} and \cite{record_biased2}, respectively. It would be interesting to see if the methods used in those papers can be utilized to calculate the permuton of the fixed point biased distribution.


\acknowledgements
\label{sec:ack}
The second author would like to thank Michael Howes for pointing out the relevance of the Mallows model with Hamming distance and for suggesting the reference \cite{Mallows_Hamming}.

\bibliographystyle{plain}
\bibliography{permutations_bib}

\begin{thebibliography}{10}

\bibitem{record_biased}
Nicolas Auger, Mathilde Bouvel, Cyril Nicaud, and Carine Pivoteau.
\newblock Analysis of algorithms for permutations biased by their number of
  records.
\newblock In {\em Proceedings of the 27th {I}nternational {C}onference on
  {P}robabilistic, {C}ombinatorial and {A}symptotic {M}ethods for the
  {A}nalysis of {A}lgorithms---{A}of{A}'16}, page~12. Jagiellonian Univ., Dep.
  Theor. Comput. Sci., Krak\'{o}w, 2016.

\bibitem{q_enumeration}
Cyril Banderier, Markus Kuba, Stephan Wagner, and Michael Wallner.
\newblock Composition schemes: {$q$}-enumerations and phase transitions in
  {G}ibbs models.
\newblock In {\em 35th {I}nternational {C}onference on {P}robabilistic,
  {C}ombinatorial and {A}symptotic {M}ethods for the {A}nalysis of
  {A}lgorithms}, volume 302 of {\em LIPIcs. Leibniz Int. Proc. Inform.}, 
  Art. No. 7, 18. Schloss Dagstuhl. Leibniz-Zent. Inform., Wadern, 2024.

\bibitem{composition_schemes}
Cyril Banderier, Markus Kuba, and Michael Wallner.
\newblock Phase transitions of composition schemes: {M}ittag-{L}effler and
  mixed {P}oisson distributions.
\newblock {\em Ann. Appl. Probab.}, 34(5):4635--4693, 2024.

\bibitem{permutation_text}
Mikl\'{o}s B\'{o}na.
\newblock {\em Combinatorics of permutations}.
\newblock Discrete Mathematics and its Applications. CRC Press, Boca Raton, FL,
  second edition, 2012.
\newblock With a foreword by Richard Stanley.

\bibitem{record_biased2}
Mathilde Bouvel, Cyril Nicaud, and Carine Pivoteau.
\newblock Record-biased permutations and their permuton limit.
\newblock arXiv:2409.01692, 2024.

\bibitem{Corsini}
Beno\^{i}t Corsini.
\newblock The height of record-biased trees.
\newblock {\em Random Structures Algorithms}, 62(3):623--644, 2023.

\bibitem{Drmota_Soria}
Michael Drmota and Mich\`ele Soria.
\newblock Images and preimages in random mappings.
\newblock {\em SIAM J. Discrete Math.}, 10(2):246--269, 1997.

\bibitem{Elizalde}
Sergi Elizalde.
\newblock Multiple pattern avoidance with respect to fixed points and
  excedances.
\newblock {\em Electron. J. Combin.}, 11(1):Research Paper 51, 40, 2004.

\bibitem{Flajolet}
Philippe Flajolet and Robert Sedgewick.
\newblock {\em Analytic combinatorics}.
\newblock Cambridge University Press, Cambridge, 2009.

\bibitem{Fulman}
Jason Fulman.
\newblock Fixed points of non-uniform permutations and representation theory of
  the symmetric group.
\newblock arXiv:2406.12139, 2024.

\bibitem{HRS_1}
Christopher Hoffman, Douglas Rizzolo, and Erik Slivken.
\newblock Pattern-avoiding permutations and {B}rownian excursion, part {II}:
  fixed points.
\newblock {\em Probab. Theory Related Fields}, 169(1-2):377--424, 2017.

\bibitem{HRS_2}
Christopher Hoffman, Douglas Rizzolo, and Erik Slivken.
\newblock Fixed points of 321-avoiding permutations.
\newblock {\em Proc. Amer. Math. Soc.}, 147(2):861--872, 2019.

\bibitem{Mallows_Hamming}
Ekhine Irurozki, Borja Calvo, and Jose~A. Lozano.
\newblock Mallows and generalized {M}allows model for matchings.
\newblock {\em Bernoulli}, 25(2):1160--1188, 2019.

\bibitem{Jones}
Brant Jones.
\newblock Weighted games of best choice.
\newblock {\em SIAM J. Discrete Math.}, 34(1):399--414, 2020.

\bibitem{permuton}
Richard Kenyon, Daniel Kr\'{a}l', Charles Radin, and Peter Winkler.
\newblock Permutations with fixed pattern densities.
\newblock {\em Random Structures Algorithms}, 56(1):220--250, 2020.

\bibitem{moment_determinacy}
Gwo~Dong Lin and Jordan Stoyanov.
\newblock Moment determinacy of powers and products of nonnegative random
  variables.
\newblock {\em J. Theoret. Probab.}, 28(4):1337--1353, 2015.

\bibitem{presortedness}
Heikki Mannila.
\newblock Measures of presortedness and optimal sorting algorithms.
\newblock {\em IEEE Trans. Comput.}, 34(4):318--325, 1985.

\bibitem{Miner_Pak}
Sam Miner and Igor Pak.
\newblock The shape of random pattern-avoiding permutations.
\newblock {\em Adv. in Appl. Math.}, 55:86--130, 2014.

\bibitem{Mukherjee}
Sumit Mukherjee.
\newblock Fixed points and cycle structure of random permutations.
\newblock {\em Electron. J. Probab.}, 21:Paper No. 40, 18, 2016.

\bibitem{Pinsky_2}
Ross~G. Pinsky.
\newblock The secretary problem with non-uniform arrivals via a left-to-right
  minimum exponentially tilted distribution.
\newblock {\em ALEA Lat. Am. J. Probab. Math. Stat.}, 20(2):1631--1642, 2023.

\bibitem{Starr}
Shannon Starr.
\newblock Thermodynamic limit for the {M}allows model on {$S_n$}.
\newblock {\em J. Math. Phys.}, 50(9):095208, 15, 2009.

\bibitem{coincidences}
Lajos Tak\'{a}cs.
\newblock The problem of coincidences.
\newblock {\em Arch. Hist. Exact Sci.}, 21(3):229--244, 1979/80.

\bibitem{descent_biased}
Paul Th\'{e}venin and Stephan Wagner.
\newblock Local limits of descent-biased permutations and trees.
\newblock arXiv:2312.11183, 2023.

\bibitem{Wendel}
J.~G. Wendel.
\newblock Note on the gamma function.
\newblock {\em Amer. Math. Monthly}, 55:563--564, 1948.

\end{thebibliography}
\label{sec:biblio}

\end{document}